\documentclass[11pt ]{amsart}
\usepackage{amsfonts}
\usepackage{mathrsfs}
\usepackage{amscd}
\usepackage{amsmath}
\usepackage{amssymb}
\usepackage{latexsym}
\usepackage{lscape}
\usepackage{xypic}

\newtheorem{theorem}{Theorem}[section]
\newtheorem{corollary}[theorem]{Corollary}
\newtheorem{lemma}[theorem]{Lemma}

\newtheorem{prop}[theorem]{Proposition}
\newtheorem{conjecture}[theorem]{Conjecture}

\theoremstyle{definition}
\newtheorem{definition}[theorem]{Definition}
\newtheorem{remark}[theorem]{Remark}
\newtheorem{example}[theorem]{Example}

\newcommand{\fg}{{\mathfrak g}}
\newcommand{\ft}{{\mathfrak t}}

\newcommand{\fb}{{\mathfrak b}}

\newcommand{\fs}{{\mathfrak s}}

\newcommand{\fa}{{\mathfrak a}}

\newcommand{\fc}{{\mathfrak c}}

\newcommand{\rW}{{\mathrm W}}

\newcommand{\bC}{{\mathbb C}}
\newcommand{\bX}{{\mathbb X}}
\newcommand{\bG}{{\mathbb G}}
\newcommand{\bZ}{{\mathbb Z}}
\newcommand{\bR}{{\mathbb R}}

\newcommand{\mE}{\mathcal{E}}
\newcommand{\mF}{\mathcal{F}}

\newcommand{\mO}{\mathcal{O}}

\newcommand{\on}{\operatorname}

\newcommand{\ra}{\rightarrow}

\newcommand{\is}{\simeq}

\newcommand{\Bun}{\on{Bun}}

\newcommand{\beq}{\begin{equation}}
\newcommand{\eeq}{\end{equation}}

\newcommand{\quash}[1]{}  
\newcommand{\nc}{\newcommand}

\newcommand{\bbZ}{{\mathbb Z}}

\newcommand{\calF}{{\mathcal F}}

\nc{\al}{{\alpha}} \nc{\be}{{\beta}} \nc{\ga}{{\gamma}}
\nc{\ve}{{\varepsilon}} \nc{\Ga}{{\Gamma}} 
\nc{\La}{{\Lambda}}

\nc{\ad}{{\on{ad}}}
\newcommand{\Ad}{{\on{Ad}}}
\nc{\aff}{{\on{aff}}} \nc{\Aff}{{\mathbf{Aff}}}
\newcommand{\Aut}{{\on{Aut}}}

\nc{\der}{{\on{der}}}

\nc{\diag}{{\on{diag}}}
\newcommand{\End}{{\on{End}}}
\nc{\Fl}{{\calF\ell}}

\newcommand{\Hom}{{\on{Hom}}}
\newcommand{\id}{{\on{id}}}
\nc{\Id}{{\on{Id}}}

\nc{\Ind}{{\on{Ind}}}
\newcommand{\Lie}{{\on{Lie}}}

\newcommand{\Res}{{\on{Res}}}
\nc{\res}{{\on{res}}}

\newcommand{\Spec}{{\on{Spec}}}

\nc{\tr}{{\on{tr}}}

\nc{\GSp}{{\on{GSp}}} \nc{\GU}{{\on{GU}}} \nc{\SL}{{\on{SL}}}
\nc{\SU}{{\on{SU}}} \nc{\SO}{{\on{SO}}}

\nc{\bU}{{\overline{U}}} \nc{\IC}{{\on{IC}}} \nc{\rH}{{\on{H}}}

\nc{\dG}{{G^\vee}} \nc{\dT}{{T^\vee}} \nc{\Zl}{{\bbZ_\ell}}
\nc{\btw}{\Bun_T^W} \nc{\tilX}{{\tilde{X}}} \nc{\tC}{\tilde{C}}
\nc{\tU}{\tilde{U}} \nc{\AJ}{{\on{AJ}}}

\newcommand{\Lg}{{^\sigma\tilde\fg}}

\quash{
\topmargin-0.5cm \textheight22cm \oddsidemargin1.2cm \textwidth14cm}
\begin{document}
\title{Vinberg's $\theta$-groups and rigid connections}
        \author{Tsao-Hsien Chen}
        \address{Department of Mathematics, University of Chicago, Chicago, IL 60637, USA.}
        \email{chenth@math.uchicago.edu}
        \address{}
         \email{}

\subjclass[2010]{17B67, 22E50, 22E57}

\maketitle

\begin{abstract}
Let $G$ be a simple complex group of adjoint type. 
In his unpublished work, Z. Yun associated to each $\theta$-group $(G_0,\fg_1)$
and a vector $X\in\fg_1$ a flat $G$-connection $\nabla ^X$ 
on the trivial $G$-bundle 
on $\mathbb P^1-\{0,\infty\}$,
generalizing the construction of Frenkel and Gross in \cite{FG}. 
In this paper we study the local monodromy of 
those flat $G$-connections and
compute the de Rham cohomology of $\nabla^X$
with values in the adjoint representations of $G$. 
In particular, we show that in many cases the connection 
$\nabla^X$ is cohomologically rigid.

\end{abstract}
 \setcounter{tocdepth}{1} \quash{\tableofcontents}
\section{Introduction}
\subsection{The goal}
Let $G$ be a simple complex algebraic group of adjoint type.
Motivated by Langlands correspondence, Frenkel and Gross \cite{FG} 
constructed a 
flat $G$-connection $\nabla$ on the trivial $G$-bundle on
$\mathbb P^1-\{0,\infty\}$ with following remarkable properties:
\begin{enumerate}
\item $\nabla$ has a regular singularity at $0$, and 
the residue is a regular nilpotent element in the Lie algebra $\fg$ of 
$G$.
\item $\nabla$ has an irregular singularity at $\infty$ with slope $1/h$,
where $h$ is the Coxeter number of $G$ (see \cite[\S5]{FG} or \cite[\S2.2]{CK}) for the definition of 
slope). 
\item 
$\nabla$ is \emph{cohomologically rigid}, i.e.,
we have $H^*(\mathbb P^1,j_{!*}\nabla^{\Ad})=0$,  
here $\nabla^\Ad$ is the $D$-module defined by the connection $\nabla$ with
values in the adjoint representation of $G$ and 
$j_{!*}\nabla^{\Ad}$ is the intermediate extension of the $D$-module $\nabla^\Ad$ to 
$\mathbb P^1$ along $j:\mathbb P^1-\{0,\infty\}\ra\mathbb P^1$.

\end{enumerate}
The construction used the $\theta$-group $(G_0,\fg_1)$ studied by Vinberg and his school, 
which comes from a $\bZ/h\bZ$-grading on $\fg$.

In his unpublished work, Z.Yun generalized the Frenkel-Gross's construction to \emph{all} $\theta$-groups. More precisely, 
starting with a $\theta$-group $(G_0,\fg_0)$, he constructed a family of flat $G$-connections $\nabla^X$ on $\mathbb P^1-\{0,\infty\}$
parametrizing by vectors $X\in\fg_0$. We called $\nabla^X$ the $\theta$-connection associated to $X\in\fg_1$.

The goal of this paper is to study properties of $\theta$-connections $\nabla^X$. 
In more detail, recall that each $\theta$-group $(G_0,\fg_1)$ corresponds to 
a torsion automorphism of $\fg=\Lie G$, which we also denote by 
$\theta\in\Aut(\fg)$. Let $\sigma$ be the image of $\theta$ in $\on{Out}(\fg)$, the group of outer automorphism of $\fg$.
We establish the following properties of $\nabla^X$, parallel to the properties (1), (2) and (3) above:
\begin{enumerate}
\item $\nabla^X$ has a regular singularity at $0$, and 
the residue $\Res(\nabla^X)$ is a nilpotent element in the Lie algebra $\fg^\sigma$. Moreover, for generic vectors $X\in\fg_1$ the residues $\Res(\nabla^X)$ lie in a single nilpotent orbit in $\fg^\sigma$.
\item $\nabla^X$ has an irregular singularity at $\infty$ with slope $e/m$ for any semi-simple $X\in\fg_1$.   
Here $m$ (resp. $e$) is the order of $\theta\in\Aut(\fg)$ (resp. $\sigma\in\on{Out}(\fg)$).
\item 
Assume $\theta$ is \emph{stable} with normalized \emph{Kac coordinates}
$s_0=1$ (see \S\ref{setup} for the definition of stable automorphism and 
normalized Kac coordinates). Then for any stable element $X\in\fg_1$ the 
connection $\nabla^X$ is \emph{cohomologically rigid}.
\end{enumerate}

\subsection{}
Let me explain briefly how these properties are obtained in the untwisted case $\sigma=id$.
Properties (1) and (2) follow basically from the construction. 
To prove property (3), we need to show that the cohomology groups
$H^*(\mathbb P^1,j_{!*}\nabla^{X,\Ad})$ vanishes (see Definition \ref{coh-rigid}). 
Here we follow an argument in \cite{FG}. First, it follows from a general result in \cite[\S 7]{FG} that
those cohomology groups
are isomorphic to kernels of certain $\bC$-linear maps
on the loop algebra $\tilde\fg=\fg[t,t^{-1}]$. 
So we reduce to show that kernels of those maps vanish.
In the case of \cite{FG},
the authors observe that if we 
consider the
\emph{principal grading} on the loop algebra $\tilde\fg$
(see Example \ref{principal grading}), 
then in terms of a homogeneous basis of $\tilde\fg$ (with respect to the principal grading), the relevant maps become 
more tractable and 
the desired cohomology vanishing follows from Kac's Theorem on
\emph{principal}
Heisenberg subalgebras of affine Kac-Moody algebras. Now our observation is 
that for general $\theta$-connections $\nabla^X$, if we consider the 
\emph{Kac grading} on $\tilde\fg$ corresponding to the 
torsion automorphism $\theta$ (see \S\ref{KMP}), then the relevant maps
again become more tractable
and the desired cohomology vanishing
follows from 
the results 
in 
\cite{Kac,V,RLYG,RY}
about gradings on Lie algebras (known as Vinberg's theory of $\theta$-groups), and 
a generalization of 
Kac's Theorem to 
\emph{general}
Heisenberg subalgebras of affine Kac-Moody algebras (see Proposition \ref{non-degenerate}
and Remark \ref{genreal Heisenberg}).

\subsection{Relation with \cite{Yun} and ramified geometric Langlands}
In \cite{Yun}, starting with a stable torsion automorphism 
$\theta\in\Aut(\check\fg)$ for the Langlands dual of $\fg$
and a stable functional $\phi\in\check\fg_1^{*,s}$,
the author constructed a remarkable $\ell$-adic $G$-local system 
$\on{KL}_{G}(\phi)$ on $\mathbb P^1-\{0,\infty\}$, which generalized 
his early work in \cite{HNY} with Heinloth and Ng$\hat{\on o}$.   
This $\ell$-adic local system is tamely ramified at $0$
and ramified at $\infty$.
He furthermore described the monodromy of $\on{KL}_{G}(\phi)$ at $0$ 
and conditionally deduced the cohomologically rigid of $\on{KL}_{G}(\phi)$
(see \cite[Theorem 4.7 and Proposition 5.2]{Yun}).
The construction can
carry out over the complex number with $\ell$-adic sheaves replaced
by $D$-modules. Thus, starting with a stable automorphism 
$\theta$ of $\check\fg$ and a stable function $\phi\in\check\fg_1^{*,s}$, we get a flat $G$-connection $\on{KL}_{G}(\phi)_{dR}$ on $\mathbb P^1-\{0,\infty\}$. The result of this paper gives strong evidence 
of the following conjecture:
\begin{conjecture}[Z.Yun and \cite{HNY} Conjecture 2.14]\label{Yun's conj}
There is a bijection between the set of stable linear functions $\phi\in\fg_1^{*,s}$ and the 
set of stable vectors $X\in\check\fg_1^s$, such that whenever $\phi$ corresponds to $X$
under this bijection, there is a natural isomorphism between 
$\sigma$-twisted flat $\check G$-connection on $\bG_m$
\[\on{KL}_{\check G}(\phi)\is\nabla^X.\]
\end{conjecture}

The solution of the conjecture above will provide 
many interesting examples of geometric Langlands correspondences 
with wild ramifications. 
When $m=h$ is the Coxeter number, i.e., in the
Frenkel-Gross case (see \S\ref{FG case}), the conjecture above was proved in \cite{Zhu} 
using a ramified version of Beilinson-Drinfeld's work on quantization of 
Hitchin's integrable systems. We plan to extend the methods in \cite{Zhu} to more general 
stable automorphisms.

\subsection{}
The paper is organized as follows.
In \S\ref{grading} we give a review of Vinberg's theory of $\theta$-groups. 
In \S\ref{KMP} we recall Kac gradings for loop algebras 
and Kac's theories on automorphism of loop algebras and Heisenberg subalgebras of 
affine Kac-Moody algebras. In \S\ref{Yun's connection} we recall Yun's 
construction of $\theta$-connections $\nabla^X$. 
We compute the residue of $\nabla^X$ at $0$ 
and the slope and irregularity at $\infty$. 
In \S\ref{main result} we prove the main results of this paper:
In Theorem \ref{rigid}, assume the $\theta$-group $(G_0,\fg_1)$ is regular, i.e., $\fg_1$ contains 
regular semi-simple elements, we compute
the de Rham cohomology of the $\theta$-connections with values in the adjoint representation.
In Theorem \ref{untwisted}, 
assume the $\theta$-group is stable and with normalized Kac coordinates $s_0=1$,
we establish the cohomological rigidity of $\nabla^X$.
Finally, in \S\ref{examples} we give several examples of $\theta$-connections.

\subsection*{Acknowledgement}
The author is grateful to Z. Yun for allowing him to use his unpublished result and for
many helpful discussions. 
He also thanks X. Zhu and M.
Kamgarpour for inspiring conversations.
The main part of this work was accomplished when the author was visiting the 
Max Planck Institute for Mathematics in Bonn. He thank the institution for the wonderful working atmosphere.

\section{Gradings on simple Lie algebras}\label{grading}

\subsection{Notation}
Let $\fg$ be the Lie algebra of a simple complex algebraic group $G$ of adjoint type.
Let $\ell$ be the rank of $\fg$.
Let $B$ be a Borel subgroup of $G$ and let $T\subset B$ be maximal torus.
We denote by $\fg=\Lie G$, $\ft=\Lie T$ and $\fb=\Lie B$. 
We let $\Ad:G\ra\Aut(\fg)$ (resp. $\ad:\fg\ra\End(\fg)$) denote the adjoint representation of 
$G$ (resp. $\fg$).
For any element 
$x\in\fg$ we denote by $\fg^x$ the kernel of $\ad(x)$. 

Let $\bX$ (resp. $\check\bX$) be the weight lattices (resp. coweight lattices) of $T$, and $R$ (resp. $\check R$)
be the set of roots (resp. co-roots) of $T$ in $G$.
We fixed a pinning $(\bX,R,\check\bX,\check R,\{E_i\})$, where $E_i\in\fg$ is a root vector for the simple 
roots $\alpha_i\in\Delta$. 

Let $\Aut(R)$ be the subgroup of $\Aut(\bX)$ preserving $R$ and let $\Aut(R,\Delta)$
be the subgroup of $\Aut(R)$ preserving $\Delta$.
The choice of pinning induced an isomorphism $\Aut(\fg)=G\rtimes\Aut(R,\Delta)$. 
Let $\exp:V:=\check\bX\otimes\bR\ra T$ be the exponential map given by
$\alpha(\exp(x))=e^{2\pi i\alpha(x)}$, for all $\alpha\in\bX$. 

For any $m\in\bZ^\times$, we let $\xi_m=e^{\frac{2\pi i}{m}}$.
For any $\bC$-vector space $V$, we denote by $V^*=\Hom(V,\bC)$ the dual of $V$.

\subsection{Affine simple roots}\label{setup}\label{affine simple roots}
In this subsection we collect some basic definitions and properties of twisted affine Kac-Moody algebra and 
affine simple roots. For details, see \cite[\S8]{Kac}.

Let $\sigma\in\Aut(R,\Delta)$ and let $e$ be the order of $\sigma$. 
We have $e=1$ or $e=2$ (type $A$, $D$, $E_6$) or $e=3$ (type $D_4$).
Consider the affine Kac-Moody algebra 
$\hat\fg=\tilde\fg\oplus\bC K\oplus\bC\mathrm d$. Here 
$\tilde\fg\oplus\bC K$ is the universal central extension of 
the loop algebra $\tilde\fg:=\fg[t,t^{-1}]$
and $[\mathrm d,t^i\otimes x+K]=it^i\otimes x$. The
automorphism $\sigma$ extends to an automorphism of $\hat\fg$ by
\[\sigma(t^i\otimes x)=\xi_e^{-i}t^i\otimes\sigma(x),\ \ \sigma(K)=K,\ \ \sigma(\mathrm d)=\mathrm d.\]
The fixed point subalgebra 
\[^\sigma\hat\fg:={^\sigma}\tilde\fg\oplus\bC K\oplus\bC\mathrm d,\] 
where $^\sigma\tilde\fg:=\fg[t,t^{-1}]^\sigma$, is the 
twisted affine Kac-Moody algebra associated to $(\fg,\sigma)$. 

Set $^\sigma\hat\ft:=\ft^\sigma\oplus\bC K\oplus\bC\mathrm d$ (here $\ft^\sigma$ is the $\sigma$-fixed vectors in $\ft$) and 
define $\delta\in(^\sigma\hat\ft)^*$ by $\delta|_{\ft^\sigma\oplus\bC K}=0$, $\delta(\mathrm d)=1$. Then there is an affine root spaces decomposition of $^\sigma\hat\fg$ with respect to $^\sigma\hat\ft$
\beq\label{KM}
^\sigma\hat\fg={^\sigma}\hat\ft\oplus\bigoplus_{\alpha\in\Phi_\aff}{^\sigma}\hat\fg_\alpha
\eeq
here $\Phi_\aff^\sigma\subset(\ft^{\sigma})^*\oplus\bC\delta$ is the set of 
affine roots.
We identify affine roots 
$\Phi_\aff^\sigma$ with affine functions on $\ft^\sigma$ by 
sending $\delta$ to the constant function $1$.

We now recall the construction of affine simple roots $\Delta_\aff^\sigma\subset
\Phi_\aff^\sigma$.
Let $R/\sigma$ (resp. $\Delta/\sigma$) be the set of orbits in $R$ (resp, $\Delta$) under $\sigma$. 
For any orbit $a\in R/\sigma$,
let $\beta_a$ denote the restriction to $\ft^\sigma$ of any $\alpha\in a$. Then the collection 
$R^\sigma:=\{\beta_a|a\in R/\sigma\}$
is a root system (possibly non-reduced) with 
basis $\Delta^\sigma=\{\beta_a|a\in\Delta/\sigma\}$. Let us choose a numbering
\[\Delta^\sigma=\{\beta_1,...,\beta_{\ell_\sigma}\},\]
where $\ell_\sigma$ is the number of $\sigma$-orbits on the set of simple roots $\Delta$.
We define a certain positive root $\eta$ as follows. 
If $\sigma=id$, then $\eta$ is the highest root. If $\sigma\neq id$, 
unless $(R,\sigma)$ is of type $^2A_{2n}$, $\eta$ is the highest short root of $R^\sigma$;
when $(R,\sigma)$ is of type $^2A_{2n}$, $\eta$ is twice the highest short root of 
$R^\sigma$.
There are unique positive integers $\{b_0=1,b_1,...,b_{\ell_\sigma}\}$ such that 
\beq\label{eta}
\eta=\sum_{i=1}^{\ell_\sigma} b_i\beta_i.
\eeq
The sum
$h_\sigma:=e\sum_{i=0}^{\ell_\sigma} b_i$
is the twisted Coxeter number of $(R,\sigma)$.
We define \[\beta_0=1/e-\eta.\]
Then $\Delta_\aff^\sigma:=\{\beta_0,...,\beta_{\ell_\sigma}\}$ is the set of affine simple roots
associated to $(\fg,\sigma)$. 
The set $C^\sigma=\{x\in V^\sigma|\beta_i(x)>0,\ i=0,...,\ell_\sigma\}$ is called the 
fundamental alcove. We denote by $\bar C^\sigma$ the closure of $C^\sigma$ in $V^\sigma$.
\quash{
Let $\fg=\bigoplus_{i\in\bZ/e\bZ}\fg_i$ be the 
$\bZ/e\bZ$-grading on $\fg$ defined by the automorphism $\sigma$.
We denote by $R^\sigma$ the 
roots of the simple Lie algebra 
$\fg^\sigma=\fg_0$ with respect to $\ft^\sigma$ and 
let $\Delta^\sigma=\{\beta_1,...,\beta_{\ell_\sigma}\}$
be the set of simple roots determined by $\fb^\sigma$. Here 
$\ell_\sigma$ is the number of $\sigma$-orbits on the set of simple roots.
We denote by $\Phi^\sigma_{\aff}$ the set of affine roots.
We identify affine roots with affine functions on $V^\sigma$. For any $\alpha\in\Phi_\aff^\sigma$ we denote 
by $\bar\alpha$ the linear part of $\alpha$. 
We denote by $\Phi^\sigma_{\aff}$ the set of affine roots.
We identify affine roots with affine functions on $V^\sigma$. For any $\alpha\in\Phi_\aff^\sigma$ we denote 
by $\bar\alpha$ the linear part of $\alpha$. 
}
\subsection{Normalized Kac coordinates}\label{Kac coordinates}
Let $\fg=\bigoplus_{i\in\bZ/m\bZ}\fg_{i}$ be a grading of $\fg$. The grading on $\fg$ 
corresponds to a torsion automorphism $\theta=\theta'\rtimes\sigma\in\on{Aut}(\fg)$ such that 
$\theta(v)=\xi_m^iv$ for $v\in\fg_i$. 
The automorphism $\theta$ is $G$-conjugate to one of the form $t\rtimes\sigma$ with $t\in T^\sigma$.
Thus without loss of generality, we can assume 
$\theta=t\rtimes\sigma$, $t\in T^\sigma$.

According to \cite[Proposition 8.1]{Kac} (see also \cite[\S3]{OV}), there exists $x\in\bar C^\sigma$
such that $\theta=\exp(x)\rtimes\sigma$.
Since $\theta$ has order $m$, we have 
\beq\label{kac coor}
\beta_i(x)=\frac{s_i}{m}.
\eeq
The integers $(s_i)_{i=0,...,\ell_\sigma}$ are the normalized \emph{Kac coordinates} of $\theta$ (see \cite[\S2.2]{RLYG}).
These coordinates satisfy 
\beq\label{equality}
e\sum_{i=0}^{\ell_\sigma}b_is_i=m,
\eeq
here $e$ is the order of $\sigma$ and the $b_i$ are integers mentioned earlier in 
\S\ref{affine simple roots}.

Let $\check\lambda=mx\in\check\bX^\sigma$.
The action of $\bG_m$ on $\fg$ via $\check\lambda$ give a grading $\fg=\bigoplus_{k\in\bZ}\fg(k)$ and 
each $\fg_{i}$ decomposes as 
\begin{equation}\label{decomp}
\fg_{ i}=\bigoplus_{k\in\bZ}\fg_{i}(k).
\end{equation} 
\begin{lemma}
In above decomposition,
we have 1)
$\fg_{i}(k)=0$ unless $k\equiv i\ \on{mod}\frac{m}{e}$ 
and 2)
$-m+es_0\leq k\leq m-es_0$.

\end{lemma}
\begin{proof}
Let $0\neq v\in\fg_i(k)$. Then we have 
$\theta(v)=\xi_m^iv$. On the other hand,
since $\theta=\check\lambda(\xi_m)\rtimes\sigma$, we have 
$\theta(v)=\sigma(\xi_m^kv)=\xi_e^l\xi_m^kv=(\xi_m)^{\frac{ml}{e}+k}v$ for some $l\in\bZ$. This implies $k\equiv i\ \on{mod}\frac{m}{e}$.
For part 2), it is enough to show that 
$\beta(\check\lambda)\leq m-es_0$, where $\beta$ is the highest root 
of $\fg$ (here we regard $\beta$ as an element in $(\ft^\sigma)^*$). 
For this, observe that we have 
$e\eta-\beta\in\sum_{i=1}^{\ell_\sigma}\bZ_{\geq 0}\beta_i$\footnote{This is obvious when $\sigma=id$,
for $\sigma\neq id$ one can use the table 1 in \cite{RLYG} to check it.}, here $\eta$ is the root introduced in 
(\ref{eta}). Hence 
\[\beta(\check\lambda)\leq e\eta(\check\lambda)=
e\sum_{i=1}^{\ell_\sigma}b_i\beta_i(\check\lambda)\stackrel{(\ref{kac coor})}=
e\sum_{i=1}^{\ell_\sigma}b_is_i\stackrel{(\ref{equality})}=m-es_0.\]
We are done.
\quash{
$\langle\check\lambda,\beta\rangle\leq m-es_0$, where $\beta$ is the highest root 
of $\fg$. For this, observe that we have 
$e\eta\geq\beta$, here $\eta$ is the root introduced in 
(\ref{eta}). Hence 
\[\langle\check\lambda,\beta\rangle\leq\langle\check\lambda,e\eta\rangle=
e\langle\check\lambda,\sum_{i=1}^{\ell_\sigma}b_i\beta_i\rangle\stackrel{(\ref{kac coor})}=
e\sum_{i=1}^{\ell_\sigma}b_is_i\stackrel{(\ref{equality})}=m-es_0.\]}

\end{proof}
\subsection{The $\theta$-group}
Let $G_0$ be the reductive subgroup of $G$ with Lie algebra $\fg_0$.
There are natural actions of $G_0\rtimes\sigma$ on $\fg_{i}$.
The pair $(G_0,\fg_1)$ is called 
$\theta$-group in the terminology of the Vinberg school.

A grading $\fg=\bigoplus_{i\in\bZ/m\bZ}\fg_i$ (resp. a torsion automorphism 
$\theta\in\Aut(\fg)$, resp. a $\theta$-group $(G_0,\fg_1)$) is called \emph{regular}, if $\fg_1$ contains a regular semi-simple element;
\emph{stable} if $\fg_1$ contains a stable element (recall that an element $v$ is called stable if $G_0$-orbit of $v$ is closed and the 
stabilizer in $G_0$ is finite).  
According to \cite[\S 5.3]{RLYG}, a vector $v\in\fg_{1}$ is stable if only if $v$ is a regular semi-simple elements of $\fg$ and the action of $\theta$ on the Cartan sub-algebra centralizing $v$ is \emph{elliptic}, i.e.
$Z_{\fg_0}(v)=0$. We denote by $\fg_1^{r}$ (resp. $\fg_1^s$) the open set of regular semi-simple
(resp. stable) elements.

For future reference, we include a lemma about structure of $\fg_0$:
\begin{lemma}[]\label{positive}
\
\begin{enumerate}
\item $\fg_0$ is the reductive subalgebra with Cartan subalgebra $\ft^\sigma$
and the system of simple roots $\Delta^\sigma_0=\{\bar\beta_i|\beta_i\in\Delta_\aff^\sigma, s_i=0\}$.
Here $\bar\beta_i$ denotes the linear part of $\beta_i$.
\item If $s_0\neq 0$, we have $\fg_0=\fg_0\cap\fg(0)=\fg(0)^\sigma$.
\end{enumerate}
\end{lemma}
\begin{proof}
Part (1) is proved in \cite[Proposition 8.6]{Kac} (see also \cite[\S3.11]{OV}).
For 
part (2) we first note that $\fg_0\supset\fg_0\cap\fg(0)=\fg(0)^\sigma$. Since 
$\fg(0)^\sigma$ is the Levi subalgebra of $\fg^\sigma$ 
with system of simple roots $\{\bar\beta_i|\beta_i\in\Delta_\aff^\sigma, i\neq 0,s_i=0\}$, part (1) implies 
$\dim\fg_0=\dim\fg(0)^\sigma$. The claim follows.

\end{proof}
\section{Kac gradings on loop algebras}\label{KMP}
\subsection{}
Let $^\sigma\tilde\fg=\fg[t,t^{-1}]^\sigma$ be the twisted loop algebra. 
The decomposition in (\ref{KM}) induces a roots space decomposition of the twisted loop algebra $^\sigma\tilde\fg=\bigoplus_{\alpha\in\Phi_\aff^{\sigma}\cup\{0\}}{^\sigma\tilde\fg}_\alpha$ (here we set $^\sigma\tilde\fg_0=\ft^\sigma$).
For any $x\in V^\sigma$, Kac has introduced 
a $\bZ$-grading 
\begin{equation}\label{Z grading}
^\sigma\tilde\fg=\bigoplus_{i\in\bZ}{^\sigma\tilde\fg_{x,i}}.
\end{equation}
Explicitly, we have 
\[^\sigma\tilde\fg_{x,i}=\bigoplus_{\alpha\in\Phi_\aff^\sigma\cup\{0\},\ \alpha(x)=\frac{i}{m}}{^\sigma\tilde\fg_\alpha}.\]
We called the $\bZ$-grading in (\ref{Z grading}) the \emph{Kac grading} associated to $x$.

\begin{example}[Principal grading]\label{principal grading}
Consider the case $\theta=\exp(x)$ where $x=\frac{\check\rho}{h}$, $\check\rho(\alpha_i)=1$ for $i=1,...,\ell$.
In this case the corresponding Kac grading can be described as follows.
Let $E_i$, $H_i$, $F_i$, $i=1,...,\ell$ be a Chevalley basis of $\fg$
such that $E_i\in\frak n$, $H_i\in\ft$, $F_i\in\bar{\frak n}$. Then $\tilde\fg=\fg[t,t^{-1}]$ has
Kac-Moody generators $e_i=E_i\otimes 1$, $h_i=H_i\otimes 1$, $f_i=F_i\otimes 1$, $i=1,...,\ell$;
$e_0=F_{-\beta}\otimes t^{}$,
$f_0=E_{\beta}\otimes t^{-1}$, here $\beta$ is the highest root and 
$E_{\beta}$ (resp. $F_{-\beta}$) is a generator of the root space $\fg_{\beta}$ (resp. $\fg_{-\beta}$).
Now the grading on $\tilde\fg$ is given by $\on{deg}(e_i)=-\on{deg}(f_i)=1$, $\on{deg}(h_i)=0$.
We call this grading the \emph{principal grading} (see \cite[\S 14]{Kac}).
\end{example}

\begin{remark}
In \cite[\S 1.3]{Kac}, Kac called the $\bZ$-grading in (\ref{Z grading}) the gradation of type $(s_0,...,s_{\ell_\sigma})$.
\end{remark}

\subsection{}\label{Kac}
We preserve the setup in \S\ref{Kac coordinates}.
Let $\theta=\exp(x)\rtimes\sigma\in\Aut(\fg)$, $x\in\bar C^\sigma$ and 
let $\fg=\bigoplus_{i\in\bZ/m\bZ}\fg_i$ be the grading on $\fg$ defined by the automorphism $\theta$.
Let $u=t^{\frac{e}{m}}$, and 
consider the following 
$\bZ$-graded Lie algebra 
\[\fg(\theta,m):=\bigoplus_{i\in\bZ} u^{i}\fg_{i}\subset\fg[u^{},u^{-1}],\]
with $\fg(\theta,m)_i:=u^{i}\fg_{i}$.
In his book, Kac proved the following.
\begin{theorem}[See \cite{Kac}, Theorem 8.5]\label{auto}
\
\begin{enumerate}
\item 
Let $\check\lambda=mx\in\check\bX^\sigma$.
The automorphism $\Ad(\lambda(u^{-1})):\fg[u,u^{-1}]\stackrel{}\is\fg[u,u^{-1}]$ induces an isomorphism
\[\Phi:\fg(\theta,m)\is\Lg.\]

\item Under the isomorphism $\Phi$, the $\bZ$-grading of $\fg(\theta,m)$
becomes to the Kac-Moy-Prasad grading of $\Lg$ associated to $x$.

\item Under the isomorphism $\Phi$, the derivation $u\partial_u$ on $\fg(\theta,m)$ becomes the 
derivation $D=\frac{m}{e}t\partial_t+\ad\check\lambda$ on $\Lg$. 
\item
The invariant form 
 $\langle u^ix,u^jy\rangle=\delta_{i,-j}(x,y)_{Kill}$ on $\fg[u,u^{-1}]$
induced an invariant from $\langle,\rangle_\theta$ (resp. $\langle,\rangle_\sigma$) on $\fg(\theta,m)$ (resp. $\Lg$) which is 
compatible with the grading, i.e., we have $\langle v,w\rangle_\theta=0$ (resp. $\langle v,w\rangle_\sigma$=0)
for $v\in\fg(\theta,m)_i$, $w\in\fg(\theta,m)_j$, $i+j\neq 0$ (resp $v\in\Lg_{x,i}$, $w\in\Lg_{x,j}$, $i+j\neq 0$).

\end{enumerate} 
\end{theorem}

We have the following corollary:
\begin{corollary}[\cite{Kac}, \cite{RY}]\label{positive 1}
\
\begin{enumerate}
\item For each $i=0,1,..,{m-1}$, 
there is a canonical isomorphism
\[{^\sigma\tilde\fg_{x,i}}=\bigoplus_{k} t^{\frac{e(i-k)}{m}}\fg_i(k)\is\fg_i.\]
where the sum is over $-m+es_0\leq k\leq m-es_0,\ k\equiv i\on{mod}\frac{m}{e}$.
\item If $i>0$, all powers $t^{\frac{e(i-k)}{m}}$ appearing in the above sum are positive, i.e., for $0<i<m$, $\fg_i(k)=0$ unless $-m+es_0\leq k\leq i,\ k\equiv i\on{mod}\frac{m}{e}$.
\end{enumerate}
\end{corollary}
\begin{proof}
Since $\fg(\theta,m)_i=u^i\fg_i$ for $i=0,1,..,{m-1}$,
result in \S\ref{setup} and above Theorem implies \[{^\sigma\tilde\fg_{x,i}}=\Phi(u^i\fg_i)=
\Ad(\lambda(u^{-1}))(u^i\fg_i)=
\bigoplus_{k} t^{\frac{e(i-k)}{m}}\fg_i(k),\]
here the sum is over $-m+es_0\leq k\leq m-es_0,\ k\equiv i\on{mod}\frac{m}{e}$.
Part (1) follows.
Since $x$ is in the fundamental alcove $\bar C^\sigma$, direct calculation 
shows that, for $i>0$, $\Lg_{x,i}$ is contained in $\Lg\cap\fg[t]$. Part (2) follows.

\end{proof}
Let us assume $\theta$ is regular and let $X\in\fg_{1}^{r}$ be a regular semi-simple element.
Consider \[p_1=\Phi(uX)\in\Lg_{x,1}.\] 
We have the following generalization of \cite[Proposition 3.8]{Kac1}
\begin{prop}\label{non-degenerate}
Let $\fa=\on{Ker}(\ad p_1)$ and $\fc=\on{Im}(\ad p_1)$. We have 
\begin{enumerate}
\item The twisted loop algebra $\Lg$ has an orthogonal decomposition
$\Lg=\fa\oplus\fc$
with respect to the invariant form $\langle,\rangle_\sigma$
in Theorem \ref{auto}. 
\item The Lie subalgebra $\fa$ is commutative.
With respect to the Kac grading, $\fa=\bigoplus\fa_i$,  
the subspaces $\fa_i$ and $\fa_j$ are orthogonal (\emph{resp.} non degenerately paired)
with respect to the invariant form $\langle,\rangle_\sigma$ on $\Lg$ if $i+j\neq 0$ (\emph{resp.}\  $i+j=0$). 
\item Consider the Kac-Moody central extension 
$^\sigma\tilde\fg\oplus\bC K$ of $^\sigma\tilde\fg$ (cf. \S\ref{affine simple roots}).
The pre-image $\hat\fa=\fa\oplus\bC K$ of $\fa$ in $^\sigma\tilde\fg\oplus\bC K$ is a non-split 
central extension of $\fa$.
\end{enumerate}
\end{prop}
\begin{proof}
We first prove part (1) and (2). 
Since the isomorphism $\Phi:\fg(\theta,m)\is\Lg$ is compatible with the invariant forms on both side, 
it is enough to prove the 
corresponding statement for 
$\fg(\theta,m)$. 
Let $\fs=\bigoplus_{i\in\bZ/m}\fs_{i}$ be the centralizer of $X$ in 
$\fg$ and $\fb=\bigoplus_{i\in\bZ/m}\fb_{i}$ be its orthogonal complement with respect to the Killing form $(,)_{Kill}$. Consider $\fa'=\on{Ker}(\ad(uX))$ and its orthogonal complement $\fc'$ in $\fg(\theta,m)$
with respect to the invariant form $\langle,\rangle_\theta$. 
We have \[\fa'=\bigoplus_{i\in\bZ}\fa_i',\ \ \ \fc'=\bigoplus_{i\in\bZ}\fc_i',\]
where $\fa_i'=u^i\fs_{i}$, $\fc'_i=u^i\mathfrak b_{i}$.
Now since the restriction of $(,)_{Kill}$ to any Cartan subalgebra is non-degenerate and 
$\ad(X)$ in invertible on $\fb$, we have 
i) $\fb_{i+1}=[X,\fb_{i}]$ and ii) $\fs_{i}$ and $\fs_{j}$ 
are orthogonal (resp. non degenerately paired)
with respect to $(,)_{Kill}$ if $i+j\neq 0$ (resp. \  $i+j=0$). 
Part (1) and (2) follow.

For part (3), we first notice that a cocycle corresponding to the central extension 
is given by
$^\sigma\tilde\fg\times{^\sigma}\tilde\fg\ra\bC$, $(v,w)\ra$
$\langle t\partial_t(v),w\rangle_\sigma$. Thus
it is enough to show that 
for any $z\neq 0\in\fa_n$ there exists $z'\in\fa_{-n}$ such that
$\langle t\partial_t(z),z'\rangle_\sigma\neq 0$. Now observe that 
\[\langle t\partial_t(z),z'\rangle_\sigma=\langle\frac{e}{m}(D-\ad\check\lambda)(z),z'\rangle_\sigma=\langle\frac{ne}{m}z,z'\rangle_\sigma-\langle\frac{e}{m}[\check\lambda,z],z'\rangle_\sigma
,\]
here $D$ is the derivation of $\Lg$ in Theorem \ref{auto}.
Since $\fa$ is commutative by part (2) we have $\langle\frac{e}{m}[\check\lambda,z],z'\rangle_\sigma=
\langle\frac{e}{m}\check\lambda,[z,z']\rangle_\sigma=0$, hence
$\langle t\partial_t(z),z'\rangle_\sigma=\langle\frac{ne}{m}z,z'\rangle_\sigma$ and the desired claim follows 
again from part (2).
\end{proof}

\section{Yun's $\theta$-connections}\label{Yun's connection}
In his unpublished work, Z. Yun associated to each torsion automorphism $\theta\in\Aut(\fg)$ and a nonzero vector 
$X\in\fg_{1}$, a \emph{twisted}
flat $G$-connection $\nabla^X$ on the trivial $G$-bundle on $\bG_m=\mathbb P^1-\{0,\infty\}$, called the $\theta$-connection associated to 
$X$. 
In this section we shall recall his construction of
$\nabla^X$ and compute its
residue at $0$ and
slope and irregularity at $\infty$.
\subsection{Twisted flat $G$-connection and cohomological rigidity}\label{theta}
In this subsection we recall the definition of twisted flat $G$-connection 
and cohomological rigidity (see \cite{Yun1} in the setting of $\ell$-adic sheaves).

Let $C$ be a smooth curve and let $\mF$ be a $G$-bundle on $C$. Denote by $\pi_\mF:\mF\ra C$ the natural projection.
A $G$-connection $\nabla$ on $\mF$ is a $G$-equivaraint map 
$\nabla:\Omega_\mF^1\ra\pi_\mF^*\Omega_C^1$ such that the composition 
$\pi_\mF^*\Omega_C^1\stackrel{d_{\pi_\mF}}\ra\Omega_\mF^1\stackrel{\nabla}\ra\pi_\mF^*\Omega_C^1$
is equal to the identity map.
Now let $\widetilde C\ra C$ be a finite \'etale Galois cover 
with Galois group $\Gamma$. Let $\sigma:\Gamma\ra\Aut(G)$ be a homomorphism.
We define a \emph{$\sigma$-twisted} flat $G$-connection on $C$ to be a triple 
$(\mathcal F,\nabla,\delta)$ where $\mathcal F$ is $G$-bundle on $\widetilde C$, $\nabla$ is a flat $G$-connection on it,
and $\delta$ is a collection of isomorphisms $\delta_\gamma:(\mathcal F,\nabla)\is\gamma^*(\mathcal F,\nabla)$, $\gamma\in\Gamma$ satisfying the usual cocycle relations with respect to the multiplication
on $\Gamma$. 

When $\mathcal F$ is the trivial $G$-bundle, then a $\sigma$-twisted flat connection on it
may be described as an operator 
\[\nabla=d+A(t)dt,\]
where $d$ is the exterior derivative and
$A(t)dt$ is a $\Gamma$-invariant $\fg$-valued one-form on $\widetilde C$, 
with $\Gamma$ acting by deck transformation and by 
the map $\Gamma\stackrel{\sigma}\ra\Aut(G)\ra\Aut(\fg)$ on $\fg$.

Let $(\mathcal F,\nabla,\delta)$ be a $\sigma$-twisted flat $G$-connection 
on a smooth curve $C$. Then the corresponding flat vector bundle $\nabla^\Ad$
on $\widetilde C$ associated to the adjoint representation 
descends to $C$ by $\Gamma$-equivariance.
Let $\bar\nabla^\Ad$ be the flat vector bundle on $C$ after 
descent.

\begin{definition}\label{coh-rigid}
A $\sigma$-twisted flat $G$-connection $(\mathcal F,\nabla,\delta)$ over an open subset $C$
of a complete smooth curve $\bar C$ is called \emph{cohomologically rigid} if 
\[H^*(\bar C,j_{!*}\bar\nabla^\Ad)=0,\]
where $j:C\hookrightarrow\bar C$ is the inclusion, and 
$j_{!*}\bar\nabla^\Ad$ is the \emph{non-derived} push forward
of the $D$-module
$\bar\nabla^\Ad$ along $j$.\footnote{
We use the notation $j_{!*}$ because 
in this case $j_{!*}\bar\nabla^\Ad$ is also the 
intermediate extension of the 
$D$-module
$\bar\nabla^\Ad$ to $\bar C$ (see \cite[\S 5.2.2]{BBD}).}
\end{definition}
\begin{remark}
Our definition of cohomological rigidity here is stronger than the usual definition. Usually one only 
requires $H^1(\bar C,j_{!*}\bar\nabla^\Ad)=0$.
\end{remark}
\begin{remark}
For a general connected reductive group $G$, one should modify the definition above by 
replacing $\bar\nabla^\Ad$ by $\bar\nabla^{\Ad,\on{der}}$ where 
$\bar\nabla^{\Ad,\on{der}}$ is the flat vector bundle associated to the representation 
$\fg^{\on{der}}=\Lie\ G^{\on{der}}$. 
For example, consider the case 
$\sigma$ is trivial, $G=GL_n$
and $C\subset\bar C=\mathbb P^1$. Then for
an irreducible $GL_n$-connection $(\mF,\nabla)$
we have $\nabla^{\Ad,\on{der}}=\End^0(\mE)$, the flat vector bundle of traceless 
endomorphisms of $\mE:=\mF\times^{GL_n}\bC^n$.
The connection $\nabla$ is 
cohomologically rigid if and only if $H^1(\mathbb P^1,j_{!*}\nabla^\Ad)=0$, 
$H^0(\mathbb P^1,j_{!*}\nabla^\Ad)=H^2(\mathbb P^1,j_{!*}\nabla^\Ad)=\bC$, which is 
equivalent to the condition that the Euler characteristic $\chi(\mathbb P^1,j_{!*}\nabla^\Ad)=2$.
In particular, we see that our definition is compatible with the one in \cite[\S 5]{Katz1}.
\end{remark}

\subsection{Construction of $\nabla^X$}\label{construction}
We preserve the setup in \S\ref{setup}.
Let $\theta=\exp(x)\rtimes\sigma\in\Aut(\fg)=G\rtimes\Aut(R,\Delta)$ be a torsion automorphism of $\fg$.
Let $\fg=\bigoplus_{i\in\bZ/m}\fg_i$ be the corresponding grading.
Let $X\in\fg_1$ and let us write $X=\sum_{} X_k$, $X_k\in\fg_{1}(k)$
according to (\ref{decomp}). 
By Corollary \ref{positive 1},
we have $X_k=0$ unless $-m+es_0\leq k\leq 1$ 
and $k\equiv1\on{mod}\frac{m}{e}$. 
Define \[p_1=\Phi(uX)=\sum_{} X_kt^{\frac{e(1-k)}{m}}\in\Lg,\]
here $\Phi$ is the isomorphism in Theorem \ref{auto}. 
Then the $\theta$-connection associated to $X$ is the following flat $G$-connection on the trivial $G$-bundle on $\bG_m=\Spec\bC[t,t^{-1}]$
\begin{equation}\label{connection}
\nabla^X=d+p_1\frac{dt}{t}=d+\sum_{-m+es_0\leq k\leq1,\ k\equiv1\on{mod}\frac{m}{e}} X_kt^{\frac{e(1-k)}{m}}\frac{dt}{t}.
\end{equation}
Note that $\frac{e(i-k)}{m}\in\bZ$. 

The $\fg$-valued one form $p_1\frac{dt}{t}$ 
is $\sigma$-invariant, 
where $\sigma$ acts on $\bG_m$ by the formula
$t\ra \xi_e^{-1}t$ and by the pinned automorphism on $\fg$.
Therefore, by the discussion in \S\ref{theta}, we can regard $\nabla^X$ as
a \emph{$\sigma$-twisted} flat $G$-connection on the trivial $G$-bundle
on $\bG_m$, where we regard $\sigma$ 
as a map $\sigma:\mu_e=\langle\xi_e\rangle\ra\Aut(G)$ sending $\xi_e\ra \sigma$,
and $\Gamma=\mu_e$ is the Galois group of 
the finite \'{e}tale Galois cover  
$[e]:\widetilde\bG_m=\bG_m\ra\bG_m$ 
given by the $e$-th power map.

\subsection{Residue at $0$}\label{Res at 0}
Notice that $\frac{e(1-k)}{m}>0$ for $k<1$, thus 
the connection $\nabla^X$  has regular singularity at $0$ with residue $\Res(\nabla^X)=X_1\in\fg_1(1)^\sigma$.
Since $\fg(1)^\sigma$ consists of nilpotent elements of $\fg^\sigma$, the residue 
is nilpotent.

Moreover, since there are only finitely many $G^\sigma$ orbits on $\fg(1)^\sigma$ (see \cite{V}),
there is a 
dense open subset of $\fg(1)^\sigma$ which lies in a single nilpotent $G^\sigma$-orbit of $\fg^\sigma$. 
We denote this orbit by $\mO_\theta$.
Thus, for generic $X\in\fg_1$ the residue $\Res(\nabla^X)$
lies in $\mO_\theta$.

The assignment $\theta\ra\mO_\theta$ gives a well defined map
\footnote{This map and the map in (\ref{Yun's map}) are due to Z. Yun.}
\[
\{\emph{torsion automorphism of $\fg$ whose image in $\Aut(R,\Delta)$ is $\sigma$}\}
\ra\{\emph{nilpotent orbits in $\fg^\sigma$}\}.\]

We now assume $\theta$ is stable. Consider the normalized Kac 
coordinates $\{s_0,s_1,...,s_{\ell_\sigma}\}$ of $\theta$. 
If we omit $s_0$ and double the remaining Kac coordinates we 
obtain the weighted Dynkin diagram for the nilpotent orbit $\mO_\theta$.
Thus for $Y$ in $\mO_\theta$, we have $\dim\fg^{\sigma,Y}=\dim\fg(0)^\sigma$.
The nilpotent class $\mO_\theta$ is distinguished if and only if 
$\dim\fg(0)^\sigma=\dim\fg(1)^\sigma$.

Recall that stable torsion automorphisms $\theta$ are classified by regular elliptic $W$-conjugacy
classes in the coset $W\sigma$ (see \cite[Corollary 15]{RLYG}). We therefore get a map
\begin{equation}\label{Yun's map}
\{\emph{regular elliptic classes in $W\sigma$}\}
\ra\{\emph{nilpotent orbits in $\fg^\sigma$}\}.
\end{equation}
In the case $\sigma=\id$ and the normalized Kac coordinates satisfies $s_0=1$,
this map is studied in \cite{S} and \cite[\S7.3]{RLYG} (see \S\ref{SD} for more details). 
The relation between this map and Kazhdan-Luszitg map \cite{KL} is discussed in 
\cite[\S8.3, Remark 2]{RLYG}.

We expect that for any stable vector $X\in\fg_1^s$ we have $X_1\in\mO_\theta$.
In other worlds, we expect the conjugacy classes of the residue $\Res(\nabla^X)$, $X\in\fg_1^s$  
depends only on $\theta$ and 
is given by the map (\ref{Yun's map}).
We will verify this expectation in some examples in \S\ref{examples}.

\subsubsection{An example}\label{FG case}
We preserve the setup in example \ref{principal grading}.
Consider $\theta=\on{exp}(\check\rho/h)$, where $\check\rho$ is the half-sum of positive co-roots and $h$ is the Coxeter number.
We have $\fg_0=\ft$ and $\fg_1=\fg(1)\bigoplus\fg(-h+1)$,
$\fg(1)=\bigoplus_{i=1}^\ell\fg_{\alpha_i}$, $\fg(-h+1)=\fg_{-\beta}$. Here $\beta$ is the 
highest root.
Choosing 
a generator $E_i$ for each $\fg_{\alpha_i}$,
a generator $E_0$ for $\fg_{-\beta}$,
 and identifying 
 $\fg_1$ with $\bigoplus_{i=0}^{\ell}\bC E_i$,  
 the open subset $\fg_1^s$ of stable vectors can be identified with 
 $\fg^s_1=\{\sum c_iE_i|c_i\neq 0\ \emph{for}\  i=0,...,\ell\}$.
For any $X=\sum_{i=0}^\ell c_iE_i\in\fg_1^s$, the corresponding $\theta$-connection 
takes the form
\[\nabla^X=d+\frac{\sum_{i=1}^\ell c_iE_i}{t}dt+c_0E_0dt.\]
This is the rigid connections constructed in \cite{FG}.
The residue of $\nabla^X$ at $0$ is $N'=\sum_{i=1}^\ell c_iX_i$,
which is regular nilpotent.

\begin{remark}\label{genreal Heisenberg}
Recall that Heisenberg algebras of the Kac-Moody central extension 
$^\sigma\tilde\fg\oplus\bC K$ are parametrized, up to conjugacy, by 
$\rW$-conjugacy classes of the coset $\rW\sigma$ (see, e.g.,\cite{KP} for the case $\sigma=\on{id}$). Given $w\in\rW\sigma$, let 
$\hat\fa_w$ denote the associated Heisenberg subalgebra.
One can show that, when
$\theta$ is stable torsion automorphism, the algebra $\hat\fa$ in Proposition \ref{non-degenerate} is conjugate to the Heisenberg sub-algebra
$\hat\fa_w$ where $w$ is an element in the regular elliptic conjugacy class of $\rW\sigma$ corresponding to $\theta$.
\end{remark}

\subsection{Slope and Irregularity at $\infty$}
In this section we compute the slope and irregularity of $\nabla^X$ at $\infty$.
We adapt the definition of the slope of a connection 
on a principal $G$-bundle 
from \cite{D,FG} (see 
\cite{BS,CK} for other equivalent definitions):
a connection on a principal $G$-bundle with irregular
singularity at a point $x$ on a curve $X$ has slope $a/b>0$ at this point if the following
holds. Let $s$ be a uniformizing parameter at $x$, and pass to the extension given by
adjoining the $b$-th root of $s$: $u^b=s$. Then the connection, written using the parameter
$u$ in the extension and a particular trivialization of the bundle on the punctured disc at
$x$ should have a pole of order $a+1$ at $x$, and its polar part at $x$ should not be nilpotent.

Let us compute the slopes of $\nabla^X$ at $\infty$ using the definition above.
Consider the covering given by $t=a^{-\frac{m}{e}}$. Then the connection $\nabla^X$ becomes
\[d-\frac{m}{e}\sum_ kX_ka^{k-1}\frac{da}{a}.\]
Taking the gauge transform with $\check\lambda^{-1}(a)$, then $\Ad(\check\lambda^{-1}(a))X_i=a^{-k}X_k$, hence the connection becomes
\begin{equation}\label{cover}
d-\frac{m}{e}X\frac{da}{a^2}+\check\lambda\frac{da}{a}.
\end{equation}
Assume $X$ is semi-simple, then according to the above definition of slopes, we see that the slope of
$\nabla^X$ at $\infty$ are either $0$ or $e/m$.

Recall that any representation $V$ of $G$ gives rise to a flat vector bundle $\nabla^{X,V}$ on $\bG_m$. We compute the irregularity of the connection $\on{Irr}_\infty(\nabla^{X,V})$ at infinity 
when $X$ is semi-simple following \cite[\S13]{FG}. Since $X$ is semi-simple the 
leading term of the connection in (\ref{cover}) is diagonalizable in any representation $V$
of $G$. It implies the slopes of the connection $\nabla^{X,V}$ is either $0$ or $e/m$,
the former occurring at the zero eigenspaces of $X$ on $V$ and the later occurring 
at the non-zero eigenspaces. According to \cite[\S 1 and \S2.3]{Katz}, the irregularity 
$\on{Irr}_\infty(\nabla^{X,V})$ is equal to the sum of the slopes of the connection 
at $\infty$. This implies 
$$
\on{Irr}_\infty(\nabla^{X,V})=\frac{e}{m}(\on{dim}V-\on{dim}V^X).
$$
Assume $\nabla^{X,V}$ descends to a flat vector bundle 
$\bar\nabla^{X,V}$ via the $e$-th power map $[e]:\widetilde\bG_m\ra\bG_m$.
Then again by \cite[\S2.3]{Katz}, we have 
\begin{equation}\label{irr}
\on{Irr}_\infty(\bar\nabla^{X,V})=
\on{Irr}_\infty(\nabla^{X,V})/e=
\frac{1}{m}(\on{dim}V-\on{dim}V^X).
\end{equation}

\begin{remark}
Let me mention that the $\theta$-connections constructed by Z. Yun (in the untwisted case) 
and their slopes 
can be also constructed and computed using the theory of 
regular strata developed by C. Bremer and D. Sage \cite{BS}. 
\end{remark}

\section{Main results}\label{main result}
We preserve the setup of \S\ref{theta}.
Let $\fg=\bigoplus_{i\in\bZ/m\bZ}\fg_i$ 
be a grading of $\fg$ and let $\theta=\theta'\rtimes\sigma\in\Aut(\fg)$
be the corresponding automorphism.
Let $X\in\fg_1$ be a nonzero vector and $\nabla^X$ be the corresponding  
$\theta$-connection, which is a $\sigma$-twisted flat $G$-connection on 
$\bG_m$.

Consider the adjoint representation $\Ad$ of $G$ on its Lie algebra $\fg$. 
The corresponding flat vector bundle 
$\nabla^{X,\Ad}$ descends via the $e$-th power map $[e]:\widetilde\bG_m\ra\bG_m$
by $\sigma$-equivariance. Let $\bar\nabla^{X,\Ad}$ be the connection after descent.

Here are the main results of this note, generalizing
\cite[Theorem 1 and Proposition 11]{FG} to general $\theta$-groups:
\begin{theorem}\label{rigid}
Assume $\theta$ is regular. Then for any regular semi-simple vector
$X\in\fg_1^r$, we have 
\[H^0(\mathbb{P}^1,j_{!*}\bar\nabla^{X,\Ad})=H^2(\mathbb{P}^1,j_{!*}\bar\nabla^{X,\Ad})=0\]
and 
\beq\label{formula}
\on{dim}H^1(\mathbb{P}^1,j_{!*}\bar\nabla^{X,\Ad})=\frac{\#R}{m}-\on{dim}\fg^{\sigma,X_1}.
\eeq
Here $j:\bG_m\hookrightarrow\mathbb P^1$ is the canonical embedding
and $X_1\in\fg_1(1)^{\sigma}$ is the residue of the connection $\nabla^X$ at $0$
(see \S\ref{Res at 0}).

\end{theorem}

\begin{theorem}\label{untwisted}
Assume $\theta$ is stable and its normalized Kac coordinates satisfies $s_0=1$.
Then for any stable vector $X\in\fg_1^s$, we have 
\[H^i(\mathbb{P}^1,j_{!*}\bar\nabla^{X,\Ad})=0\]
for all $i$, that is, $\nabla^X$ is cohomologically rigid (see Definition \ref{coh-rigid}).

\end{theorem}

\section{Proofs}

\subsection{}
The key step leading the proofs of Theorem \ref{rigid}  and Theorem \ref{untwisted} 
is the computation of the cohomology groups 
$H^0(D_0^\times,\bar\nabla^{X,\Ad})$ and $H^0(D_\infty^\times,\bar\nabla^{X,\Ad})$. Here $D_0^\times=\Spec\bC((t))$ (resp. $D_\infty^\times=\Spec\bC((t^{-1}))$) is the formal punctured disc around 
$0$ (resp. $\infty$).

We first introduce some auxiliary notations that will be used in the rest of the section.
Let $\nabla^X=d+p_1\frac{dt}{t}$ be the $\theta$-connection associated to $X\in\fg_1^s$ (see \S\ref{construction}).
Recall $p_1=\sum_{} X_kt^{\frac{e(1-k)}{m}}\in\Lg$.
The connection $\nabla^X$ gives a $\bC$-linear map
\[\nabla^{X,\Ad}:\fg[[t,t^{-1}]]\ra\fg[[t,t^{-1}]]\frac{dt}{t}.\]
Let $f=\sum v_nt^n\in\fg[[t,t^{-1}]]$ be a solution to 
$\nabla^{X,\Ad}(f)=0$. The  components $v_n$ satisfy
\begin{equation}\label{solution}
nv_n+[X_1,v_n]+\sum_{-m+es_0\leq i\leq 0,\ i\equiv1\on{mod}\frac{m}{e}} [X_i,v_{a_i}]=0,
 \end{equation}
where $a_i=n-\frac{e(1-i)}{m}\in\bZ$. Notice that $a_i<n$ for all $i$. 

\subsubsection{}\label{inertia at 0}
We compute $H^0(D_0^\times,\bar\nabla^{X,\Ad})$. 
Recall $\bar\nabla^{X,\Ad}$ is the descent of $\nabla^{X,\Ad}$ along the $e$-th power map
$[e]:\bG_m\ra\bG_m$. Thus $H^0(D_0^\times,\bar\nabla^{X,\Ad})=H^0(D_0^\times,[e]^*\bar\nabla^{X,\Ad})^\sigma
=H^0(D_0^\times,\nabla^{X,\Ad})^\sigma=\on{Ker}(\nabla^{X,\Ad}:\fg((t))^\sigma\ra\fg((t))^\sigma\frac{dt}{t})$.
Let $f=\sum v_nt^n\in\fg((t))^\sigma$ be a solution to 
$\nabla^{X,\Ad}(f)=0$. If $f\neq 0$,
then there exists $b\in\bZ$ such that $v_b\neq 0$ and $v_s=0$ for $s<b$. 
We claim that $b\geq0$. Indeed, if $b<0$, 
then
equation (\ref{solution}) implies \[bv_b+[X_1,v_b]=0,\] 
which is impossible since the operator $b\cdot\Id+X_1$ is invertible (recall $X_1\in\fg(1)^\sigma$ is nilpotent). 
Thus $f=\sum v_nt^n\in\fg[[t]]^\sigma$ and $v_0\in\fg^\sigma$ lies in the kernel of $\ad(X_1)$. The equation (\ref{solution}) also implies there is a unique
solution $f=\sum_{} v_nt^n$ in $\fg[[t]]^\sigma$ for each $v_0\in\fg^{\sigma,X_1}$. Above discussion shows that 
\[H^0(D_0^\times,\bar\nabla^{X,\Ad})=\on{Ker}(\nabla^{X,\Ad}:\fg((t))^\sigma\ra\fg((t))^\sigma\frac{dt}{t})=\fg^{\sigma,X_1}.\]

\subsubsection{}\label{inertia at infinity}
We show that 
$H^0(D_\infty^\times,\bar\nabla^{X,\Ad})$ is zero.
For this, we need some 
preliminary results about solutions $f\in\fg[[t,t]]^\sigma$ to $\nabla^{X,\Ad}(f)=0$.
Let $f$ be such a solution.
If we write $f$  in its components for the 
\emph{Kac grading}: $f=\sum y_n$ where $y_n\in\Lg_{x,n}$, then we have
\[(\frac{m}{e}t\nabla^{X,\Ad})(f)=\sum_n((\frac{m}{e}t\partial_t+\ad\check\lambda)y_n+\frac{m}{e}[p_1,y_n]-[\check\lambda,y_n])dt=0.\]
Recall $p_1=\sum_{} X_kt^{\frac{e(1-k)}{m}}\in\Lg$.

Notice that the operator $\frac{m}{e}t\partial_t+\ad\check\lambda$ is exactly the 
derivation $D$ of $\Lg$ in Theorem \ref{auto} which defines the \emph{Kac grading}.
Thus we have $(\frac{m}{e}t\partial_t+\ad\check\lambda)y_n=ny_n$ and above equation gives rise to the identity 
\begin{equation}\label{equ theta}
ny_n-[\check\lambda,y_n]+\frac{m}{e}[p_1,y_{n-1}]=0
\end{equation}
for all $n\in\bZ$.
 
We have the following lemma 
\begin{lemma}[\cite{FG}, Lemma 6]\label{FG}
Suppose that $y_n$ satisfying (\ref{equ theta}) and $y_n\in\fa_n$ for some $n$.
Then $y_m=0$ for all $m\leq n$.
\end{lemma}
\begin{proof}
Assume that $y_n\neq 0$. 
In the course of the proof of Corollary \ref{non-degenerate} (part (3)),
we have shown that there exists $z\in\fa_{-n}$ such
that $\langle t\partial_t(y_n),z\rangle_\sigma\neq 0$. On the other hand,
since $y_n$ satisfies (\ref{equ theta}), we have 
\[t\partial_t(y_n)=\frac{e}{m}(D-\ad\check\lambda)(y_n)=\frac{e}{m}(ny_n-[\check\lambda,y_n])=-[p_1,y_{n-1}]\in\fc\]
and it implies $\langle t\partial_t(y_n),z\rangle_\sigma=0$. 
We get a contradiction. Hence $y_n$ must be zero.

Now, the equation (\ref{equ theta}) shows that if $y_n=0$ then
$y_{n-1}\in\fa_{n-1}$, hence, by induction that $y_m=0$ for all
$m\leq n$. 
\end{proof}

Above lemma implies $H^0(D_\infty^\times,\bar\nabla^{X,\Ad})=0$.
To see this, observe that  
\[H^0(D_\infty^\times,\bar\nabla^{X,\Ad})=\on{Ker}(\nabla^{X,\Ad}:\fg((t^{-1}))^\sigma\ra\fg((t^{-1}))^\sigma\frac{dt}{t}).\]
Let $f\in\on{Ker}(\nabla^{X,\Ad}:\fg((t^{-1}))^\sigma\ra\fg((t^{-1}))^\sigma\frac{dt}{t})$. Then 
we have $v_n=0$ for $n\gg 0$. This implies $y_n=0$ for $n\gg 0$ (recall that $y_n$ are the components of $f$ for the 
\emph{Kac grading}), hence by above lemma $y_n=0$ for all $n$.
So we must have $f=0$.

\subsection{Proof of Theorem \ref{rigid}}
According to \cite[\S8]{FG}, we have 
\[H^0(\mathbb{P}^1,j_{!*}\bar\nabla^{X,\Ad})=H^0(\bG_m,\bar\nabla^{X,\Ad}),\]
\[H^2(\mathbb{P}^1,j_{!*}\bar\nabla^{X,\Ad})=H^2_c(\bG_m,\bar\nabla^{X,\Ad}),\]
and there is an exact sequence
\[0\ra H^0(\bG_m,\bar\nabla^{X,\Ad})\ra H^0(D_0^\times,\bar\nabla^{X,\Ad})\oplus
H^0(D_\infty^\times,\bar\nabla^{X,\Ad})\ra \]
\[H^1_c(\bG_m,\bar\nabla^{X,\Ad})\ra 
H^1(\mathbb{P}^1,j_{!*}\bar\nabla^{X,\Ad})\ra0.\]
We first prove $H^0(\mathbb{P}^1,j_{!*}\bar\nabla^{X,\Ad})=H^2(\mathbb{P}^1,j_{!*}\bar\nabla^{X,\Ad})=0.$
Since $H^0(D_\infty^\times,\bar\nabla^{X,\Ad})=0$ by the result in \S\ref{inertia at infinity}, 
 $\bar\nabla^{X,\Ad}$ admits no global sections, i.e.,  
 $H^0(\mathbb{P}^1,j_{!*}\bar\nabla^{X,\Ad})=H^0(\bG_m,\bar\nabla^{X,\Ad})=0$.
Dually, $H^2(\mathbb{P}^1,j_{!*}\bar\nabla^{X,\Ad})=H^2_c(\bG_m,\bar\nabla^{X,\Ad})=
H^0(\bG_m,\bar\nabla^{X,\Ad})^*=0$. Here we used the fact adjoint representation 
$\Ad$ is self-dual, hence $(\bar\nabla^{X,\Ad})^*\is\bar\nabla^{X,\Ad}$.

Now we prove $\on{dim}H^1(\mathbb{P}^1,j_{!*}\bar\nabla^{X,\Ad})=
\#R/m-\on{dim}\fg^{\sigma,X_1}$.
Results from \S\ref{inertia at 0}, \S\ref{inertia at infinity} and above exact sequence imply
\begin{equation}\label{H_c^1}
0\ra\fg^{\sigma,X_1}\ra H^1_c(\bG_m,\bar\nabla^{X,\Ad})\ra 
H^1(\mathbb{P}^1,j_{!*}\bar\nabla^{X,\Ad})\ra0.
\end{equation}
Thus it suffices to prove that
$\on{dim}H^1_c(\bG_m,\bar\nabla^{X,\Ad})=\#R/m$.

Recall the Deligne's formula in \cite[\S6.21.1]{D} for the Euler characteristic
\[\chi_c(\bG_m,\bar\nabla^{X,\Ad}):=\sum_i (-1)^i\on{dim}H^i_c(\bG_m,\bar\nabla^{X,\Ad})
=\chi_c(\bG_m)\on{rank}(\bar\nabla^{X,\Ad})-\sum_{\alpha=0,\infty}
\on{Irr}_\alpha(\bar\nabla^{X,\Ad}).\]
Since $\chi_c(\bG_m)=0$ and $\bar\nabla^{X,\Ad}$ is regular at $0$, it implies
\[\chi_c(\bG_m,\bar\nabla^{X,\Ad})=-\on{Irr}_\infty(\bar\nabla^{X,\Ad}).\]
Using the vanishing of $H^0_c$, $H_c^2$ and the formula in line (\ref{irr}), we get
\[\on{dim}H^1_c(\bG_m,\bar\nabla^{X,\Ad})=\on{Irr}_\infty(\bar\nabla^{X,\Ad})
=\frac{1}{m}(\on{dim}\fg-\dim\fg^X).\]
Since $X$ is regular semi-simple, we have $\frac{1}{m}(\on{dim}\fg-\dim\fg^X)=\#R/m$, hence
\[\on{dim}H^1_c(\bG_m,\bar\nabla^{X,\Ad})=\#R/m.\]
This finished the proof of Theorem \ref{rigid}.

\subsection{Proof of Theorem \ref{untwisted}}
It is enough to show that $H^1(\mathbb{P}^1,j_{!*}\bar\nabla^{X,\Ad})=0$.
We begin with the following lemma:

\begin{lemma}\label{FG1}
For any solution $f=\sum v_nt^n$ of $\nabla^{X,\Ad}(f)=0$ in $\fg[[t,t^{-1}]]^\sigma$ we have 
$v_n=0$ for all $n<0$.
\end{lemma}
\begin{proof}
Write $f=\sum y_n$ in the components for the Kac grading.
When $n=0$, the equation (\ref{equ theta}) becomes
\[-[\check\lambda,y_0]+m[p_1,y_{-1}]=0.\]
Since $s_0=1$ by assumption, Lemma \ref{positive} implies 
$y_0\in\fg_{0}=\fg_0\cap\fg(0)\subset\on{ker}(\ad\check\lambda)$. Therefore above equation implies 
$y_{-1}\in\fa_{-1}$, thus by Lemma \ref{FG}, we have
$y_{n}=0$ for $n<0$, or equivalently $f=\sum_{n\geq 0} y_n$. 
On the other hand, Corollary \ref{positive 1} and the fact $\fg_0=\fg_0\cap\fg(0)$ imply $y_n\in\fg[t]$ for $n\geq0$. 
The Lemma follows. 
 \end{proof}

By \cite[\S 9]{FG}, we have 
$H^1_c(\bG_m,\bar\nabla^{X,\Ad})\is\on{Ker}(\nabla^{X,\Ad}:\fg[[t,t^{-1}]]^\sigma\ra\fg[[t,t^{-1}]]^\sigma\frac{dt}{t})$, which is equal to 
$\on{Ker}(\nabla^{X,\Ad}:\fg[[t]]^\sigma\ra\fg[[t]]^\sigma\frac{dt}{t})$ by above Lemma. The same argument as in \S\ref{inertia at 0} shows that $\on{Ker}(\nabla^{X,\Ad}:\fg[[t]]^\sigma\ra\fg[[t]]^\sigma\frac{dt}{t})=\fg^{\sigma,X_1}$. Therefore the first two terms in the 
short exact sequence (\ref{H_c^1}) 
both have dimension $\dim\fg^{\sigma,X_1}$
. This proves the vanishing of $H^1(\mathbb P^1,j_{!*}\bar\nabla^{X,\Ad})$, hence finished the proof of 
Theorem \ref{untwisted}

\section{Examples}\label{examples}
In this section we give several examples of $\theta$-connections $\nabla^X$.
In each example we write down the connection explicitly and 
check its cohomological rigidity using the formula in (\ref{formula}).
We also check that, in each case, the residue of the $\theta$-connection at $0\in\mathbb P^1$ (or rather its conjugacy classes) 
 depends only on $\theta$, hence verify our expectation in \S\ref{Res at 0}.
References for this section are 
\cite{FG,RLYG,RY}.

\subsection{S-distinguished nilpotent case}\label{SD}
Recall that a nilpotent element $N$ in $\fg$ is called \emph{distinguished} if
$\fg^N$ consists of nilpotent elements. 
Let $N\in\fg$ be a distinguished nilpotent element.
There is a co-character $\check\lambda$ such that $\Ad\check\lambda(t)N=tN$
for all $t\in\bC^\times$. This gives a grading $\fg=\bigoplus_{k=-a}^{a}\fg(k)$
where $\fg(k)=\{x\in\fg|\Ad\check\lambda(t)x=t^kx\}$.
Set $m=a+1$ and consider the inner automorphism 
$\theta_N:=\check\lambda(\xi_m)\in\Aut(\fg)$. We have $\fg_0=\fg(0)$ and $\fg_1=\fg(1)\oplus\fg(-a)$.
Following \cite[\S7.3]{RLYG}, we say that a distinguished nilpotent element 
$N\in\fg$ is \emph{S-distinghuished} if 
the automorphism $\theta_N$ is stable. According to \emph{loc. cit.},
a nilpotent element $N$ is $S$-distinguished 
if and only if there exits $E\in\fg(-a)$ such that
$N+E\in\fg_1$ is stable. 
Moreover, assume $\fg$ is of exceptional 
type, the map $N\ra\theta_N$ defines a bijection between the set
of $S$-distinguished nilpotent orbits in $\fg$ to the set
of stable inner automorphism on $\fg$ with $s_0=1$.

Let $\theta_N$ be the stable automorphism of $\fg$ corresponding to
a S-distinguished nilpotent element $N\in\fg$. 
Let $X=N+E\in\fg_1=\fg(1)\oplus\fg(-a)$ be a stable vector.
Then by the formula in (\ref{connection}), the corresponding $\theta$-connection $\nabla^X$ takes the form
\[\nabla^X=d+\frac{N}{t}dt+Edt.\]
Note that $N$ is the residue of $\nabla^X$ at zero.

Let us verify that $\nabla^X$ is cohomologically rigid, i.e., $\dim H^*(\mathbb{P}^1,j_{!*}\nabla^{X,\Ad})=0$.
By Theorem \ref{rigid}, we have $H^0=H^2=0$. Thus it remains to show  
$H^1(\mathbb{P}^1,j_{!*}\nabla^{X,\Ad})=\frac{\#R}{m}-\on{dim}\fg^{N}=0$.
To see this recall that $N$ is distinguished, thus we have $\dim\fg^N=\dim\fg(0)=\dim\fg_0$.
On the other hand, we have $\dim\fg_0=\frac{\#R}{m}$ (see \cite[Theorem 4.2]{P}).
Result follows.

\quash{
\footnote{Since the Kac coordinates of $\theta_N$ 
not necessarily satisfies $s_0=1$, the rigidity does not follow from
Theorem \ref{untwisted}}.}

\quash{
\subsubsection{Frenkel-Gross case}\label{FG case}
We preserve the setup in \S\ref{setup}.
Consider the regular nilpotent element $N=\sum_{i=1}^l E_i$ in\ $\fg$. By \cite[Corollary 6.4]{Kos}, it is S-distinguished.
We can take the co-character to be $\check\lambda=\check\rho$ 
and $\theta_N=\check\rho(\xi_h)\in\Aut(\fg)$, where $\check\rho$ is the half-sum of positive co-roots and $h$ is the Coxeter number.

 We have $\fg_0=\ft$ and $\fg_1=\fg(1)\bigoplus\fg(-h+1)$,
$\fg(1)=\bigoplus_{i=1}^\ell\fg_{\alpha_i}$, $\fg(-h+1)=\fg_{-\beta}$. Here $\beta$ is the 
highest root.
Choosing a generator $E_0$ for $\fg_{-\beta}$
 and identifying 
 $\fg_1$ with $\bigoplus_{i=0}^{\ell}\bC E_i$,  
 the open subset $\fg_1^s$ of stable vectors can be identified with 
 $\fg^s_1=\{\sum c_iE_i|c_i\neq 0\ \emph{for}\  i=0,...,\ell\}$.
For any $X=\sum c_iE_i\in\fg_1^s$, the corresponding $\theta$-connection 
takes the form
\[\nabla^X=d+\frac{\sum_{i=1}^\ell c_iE_i}{t}dt+c_0E_0dt.\]
This is the rigid connections constructed in \cite{FG}.
Note that the residue of $\nabla^X$ at $0$ is $N'=\sum_{i=1}^\ell c_iX_i$,
which is regular nilpotent.
}

 \subsubsection{Type $G_2$}
Let $\fg$ is the simple Lie algebra of type $G_2$.
Let $\alpha_1,\alpha_2$ be the simple root of $\fg$, where $\alpha_2$ is the short root. 
Consider the automorphism $\theta=\check\lambda(\xi_3)$, 
where $\check\lambda=\check\omega_1$ is the fundamental co-weight dual to $\alpha_1$ and 
$\xi_3$ is a $3$-th primitive root of unity. According to \cite{RLYG}, $\theta$ is a stable inner automorphism of order $3$ with normalized Kac coordinates
\[1\ 1\Rrightarrow 0.\]
Observe that if we omit $s_0$ and double remaining the Kac coordinates we 
obtain 
\[2\Rrightarrow 0,\]
which is the weighted Dynkin diagram for the nilpotent orbit $G_2(2)$.
This implies $\theta$ is equal to $\theta_N$ in \S\ref{SD} for some $N\in G_2(2)$. 

We have $G_0=\on{GL}_2(\bC)$ and $\fg_1=\fg(1)\oplus\fg(-2)$
, $\fg(1)=\bigoplus_{k=0}^3\fg_{\alpha_1+k\alpha_2}$, 
$\fg(-2)=\fg_{-2\alpha_1-3\alpha_2}$.
As a representation of $G_0=\on{GL}_2(\bC)$, we have 
\begin{equation}\label{iso}
\fg(1)\is\on{det}^2\otimes P_3,\ \  \fg(-2)\is\on{det}^{-1}\otimes P_0,
\end{equation}
where $P_d$ is the space of homogeneous polynomials of degree $d$ on 
$\bC^2$, with the natural action of $G_0=\on{GL}_2(\bC)$.
Choosing coordinates, we regard a vector $X\in\fg_1$ as a 
pair $(f,z)$, where $f=f(x,y)$ is a binary cubic polynomials over $\bC$ and 
$a\in\bC$. According to \cite[\S7.5]{RY}, we have $(f,z)\in\fg_1^s$ if and only if $z\neq0$ and $f$ has three
distinct roots in the projective line. For any $X\in\fg_1^s$, let us write $X=X_1+X_{-2}$ according to the decomposition $\fg_1=\fg(1)\oplus\fg(-2)$. 
The corresponding 
$\theta$-connection takes the form
\[\nabla^X=d+\frac{X_1}{t}dt+X_{-2}dt.\]
We claim that the residue $X_1$ is in the subregular nilpotent orbit 
$G_2(2)$. In particular, the conjugacy classes of the residue $\Res(\nabla^X)$ is independent of the choice $X\in\fg_1^s$.
To prove the claim, observe that the intersection of $G_2(2)$ with $\fg(1)$ is open dense. Thus to show that 
$X_1$ is in $G_2(2)$ it is enough to show that $\dim\Ad G_0(X_1)=\dim\fg(1)=4$. 
But it follows from the fact that the centralizer $Z_{G_0}(X_1)$ of $X_1$ in $G_0$ is the symmetric group $S_3$
(permuting the roots of $f$, where $f$ is the binary cubic polynomial corresponding to $X_1$ under the isomorphism (\ref{iso})), hence $\dim\Ad G_0(X_1)=\dim G_0=4$.

 \subsection{Type $^2A_{2n}$}
Let $\fg=\mathfrak{sl}_{2n+1}(\bC)$ ($n\geq 1$). Let $\sigma$ be a pinned automorphism of $\fg$.
We define $\theta=\check\rho(-1)\rtimes\sigma\in\Aut(\fg)$. 
According to \cite{RLYG}, it is a stable involution with Kac coordinates
\[1\Rightarrow0\ 0\cdot\cdot\cdot0\ 0\Rightarrow 0.\]
Let $\fg=\fg_0\oplus\fg_1$ be the corresponding grading. 
We give a description of $\fg_0$ and $\fg_1$. Let $V$ be a vector space 
over $\bC$ of dimension $2n+1$ with basis $\{x_{-n},...,x_{-1},x_0,x_{1},...,x_n\}$.
We define an inner product $\langle,\rangle$ on $V$ by the formula
$\langle\sum a_ix_i,\sum b_ix_i\rangle=\sum_{i=-n}^n a_ib_{-i}.$
For any $X\in\mathfrak{gl}(V)$, let $X^*$ be the adjoint of $X$ with respect to this inner product.
Then under the canonical isomorphism $\mathfrak{sl}(V)\is\fg$, we have $\theta(X)=-X^*$ for any $X\in\fg$.
Thus $\fg_0\is\mathfrak{so}(V)=\{X\in\mathfrak{sl}(V)|X=-X^*\}$, $\fg_1=\{X\in\mathfrak{sl}(V)|X=X^*\}$.
Moreover, we have $\fg_1^s=\fg_1\cap\fg^{rs}$, here $\fg^{rs}$
is the open subset of regular semi-simple elements in $\fg$.

Since $m=e=2$ (recall $m$ and $e$ are the order of $\theta$ and $\sigma$), Corollary \ref{positive 1} implies
$\fg_1=\fg_1(0)$. Thus for any $X\in\fg_1$, the $\theta$-connection $\nabla^X$ has the form
\[\nabla^X=d+Xdt.\]
In particular, it is unramified at zero.

Finally, since $\fg^\sigma$ is a simple lie algebra of type $B_{n}$ we have $\dim\fg^\sigma=n(2n+1)$.
Thus for $X\in\fg_1^s$ we have 
\[\on{dim}H^1(\mathbb{P}^1,j_{!*}\bar\nabla^{X,\Ad})=\frac{\#R}{2}-\on{dim}\fg^{\sigma}=\frac{2n(2n+1)}{2}-
n(2n+1)
=0.\]

\end{document}